\documentclass[a4paper] {article}
[12pt]

\makeatletter
\renewcommand*\l@section{\@dottedtocline{1}{1.5em}{2.3em}}
\makeatother

\usepackage{amsfonts}
\usepackage{amssymb}
\usepackage[T1]{fontenc}

\usepackage{tikz}
\usetikzlibrary{calc}

\usepackage{CJK}
\usepackage{amsmath}

\usepackage{amsfonts}
\usepackage{amssymb}
\usepackage{amsthm}
\usepackage{amssymb}
\usepackage{enumerate}
\usepackage[calc]{picture}
\usepackage[all,cmtip]{xy}

\usepackage[mathscr]{eucal}
\usepackage{eqlist}

\usepackage{color}
\usepackage{abstract}
\usepackage[T1]{fontenc}

 \usepackage[top=3.0 cm, bottom=3.0cm, left=3.5cm, right= 3.5cm]{geometry}

\usepackage{cite}

\theoremstyle{plain}
\newtheorem{theorem}{Theorem}[section]
\newtheorem{proposition}[theorem]{Proposition}
\newtheorem{lemma}[theorem]{Lemma}
\newtheorem{corollary}[theorem]{Corollary}

\theoremstyle{definition}
\newtheorem{definition}{Definition}[section]

\theoremstyle{example}
\newtheorem{example}{Example}[section]

\theoremstyle{remark}
 \newtheorem{remark}{Remark}[section]

\usepackage{etoolbox}

\numberwithin{equation}{section}
\numberwithin{theorem}{section}

\begin{document}

\begin{center}
{\Large {\textbf {A K\"{u}nneth Formula for Product of Hypergraphs
}}}
 \vspace{0.58cm}

Chong Wang*\dag, Shiquan Ren*, Jian Liu*

\bigskip

\bigskip

    \parbox{24cc}{{\small
{\textbf{Abstract}.}
In this paper, based on the embedded homology groups of hypergraphs defined in \cite{h1},
we define the product of hypergraphs and prove the corresponding K\"{u}nneth formula of hypergraphs which  can be generalized to the K\"{u}nneth formula for the embedded  homology of graded subsets of chain complexes  with coefficients in a principal ideal domain.
}}
\end{center}

\vspace{1cc}



\footnotetext[3] { * first authors;   \dag   corresponding author. }

\section{Introduction}
In topology, a hypergraph can be obtained from a simplicial complex by deleting some non-maximal simplices  (cf. \cite{h1,parks}). Hypergraph is a standard mathematical network-model for many   real-world data problems.  For example,  the coauthorship network of scientific
researchers and their collaborations (cf. \cite{coa-1}),  the biological cellular networks (cf. \cite{network2}), and   the network of biomolecules
and biomolecular interactions (cf. \cite{add-1}). Hypergraph is the key hub to connect the simplicial complex in topology and graph in combinatorics,
which is worth studying in theory and application (cf. \cite{berge,h1,parks}).

Let $V_\mathcal{H}$ be a totally-ordered finite set. Let $2^V$ denote the powerset of $V$.  Let $\emptyset$ denote the empty set.  A {\it hypergraph}  is a pair $(V_\mathcal{H},\mathcal{H})$ where $\mathcal{H}$ is a subset of  $2^V\setminus\{\emptyset\}$ (cf. \cite{berge,parks}).  An element of $V_\mathcal{H}$ is called a {\it vertex}  and an element of $\mathcal{H}$ is called a {\it hyperedge}.  A  hyperedge $\sigma\in\mathcal{H}$ consisting of $k+1$ vertices  is called a  {\it $k$-dimensional hyperedge} ($k\geq 0$), denoted as $\sigma^{(p)}$ or $\sigma$ for short. Throughout this paper, we assume that each vertex in $V_\mathcal{H}$ appears in at least one hyperedge in $\mathcal{H}$.  Hence $V_\mathcal{H}$ is the union $\bigcup_{\sigma\in\mathcal{H}} \sigma$, and we simply denote a hypergraph $(V_\mathcal{H},\mathcal{H})$ as $\mathcal{H}$.

Let $\mathcal{H}$ be a hypergraph. The {\it associated simplicial complex} $\mathcal{K}_{ \mathcal{H}}$ of $\mathcal{H}$ is defined as the smallest simplicial complex that $\mathcal{H}$ can be embedded in (cf. \cite{parks}). Precisely, the set of all simplices of $\mathcal{K}_{\mathcal{H}}$ consists of all
the non-empty subsets $\tau \subseteq \sigma$, for all $\sigma \in \mathcal{H}$.

There are various (co)homology theories of hypergraphs. For example, A.D. Parks and S.L. Lipscomb \cite{parks} studied the homology of
the associated simplicial complex in 1991. F.R.K. Chung and R.L. Graham \cite{Chung} constructed certain cohomology for hypergraphs in
a combinatorial way in 1992. E. Emtander \cite{Emtander} constructed the independence simplicial complexes for hypergraphs and studied
the homology of these simplicial complexes, and J. Johnson \cite{Johnson}  applied the topology of hypergraphs to study hyper-networks
of complex systems in 2009. S. Bressan, J. Li, S. Ren and J. Wu \cite{h1}  defined the embedded homology of hypergraphs
as well as the persistent embedded homology of sequences of hypergraphs in 2019.

Let $R$ be  a principal ideal domain. Let $\mathcal{H}$ be a hypergraph. Let $R(\mathcal{H})_n$ be the finitely generated free $R$-module with generators of
 $n$-dimensional hyperedges in $\mathcal{H}$. Let $\mathcal{K}$ be a  simplicial complex such that $\mathcal{H}\subseteq \mathcal{K}$. By \cite{h1}, the $\mathit{infimum ~chain ~complex}$ and
the $\mathit{supremum ~chain ~complex}$ of $\mathcal{H}$ are  defined as
\begin{eqnarray*}
 \text{Inf}_n(R(\mathcal{H})_{\ast})= R(\mathcal{H}_n) \cap (\partial_n)^{-1}R(\mathcal{H})_{n-1}, n \geq 0
\end{eqnarray*}
and
\begin{eqnarray*}
\text{Sup}_n(R(\mathcal{H})_{\ast}) = R(\mathcal{H}_n) \cap \partial_{n+1}R(\mathcal{H})_{n+1}, n \geq 0
\end{eqnarray*}
respectively, where $\partial_{\ast}$ is the boundary maps of $\mathcal{K}$ and $\partial_{\ast}^{-1}$
denotes the pre-image of $\partial_{\ast}$. In \cite{h1},  it is proved that both $\text{Inf}_n(R(\mathcal{H})_{\ast})$ and $\text{Sup}_n(R(\mathcal{H})_{\ast})$ do not depend on the choice of the simplicial complex $\mathcal{K}$ that $\mathcal{H}$ embedded in. Therefore,
$\mathcal{K}$ can be taken as the associated simplicial complex of $\mathcal{H}$.  By \cite[Proposition~2.4]{h1},
the homologies of these two chain complexes are isomorphic, which are defined as the {\it embedded homology} of $\mathcal{H}$ and
denoted as $H_n(\mathcal{H})$.  In particular, if the hypergraph is a simplicial complex, then the embedded homology coincides with the usual homology.
Moreover, each morphism of hypergraphs from $\mathcal{H}$ to $\mathcal{H}'$ induces an homomorphism between the embedded homology $H_n(\mathcal{H};R)$ and
$H_n(\mathcal{H'};R)$ (\cite[Proposition~3.7]{h1}).

K\"{u}nneth formulas describe the homology or cohomology of a product space in terms of the homology or cohomology of the factors.
In \cite{hatcher}, Hatcher gave the classical algebraic K\"{u}nneth formula. In \cite{yau1,3}, A. Grigor'yan, Y. Lin, Y. Muranov and S.T. Yau  studied the the K\"{u}nneth
formula for the path homology (with field coefficients) of digraphs.

In this paper, we define the product ``$\boxtimes$'' of hypergraphs and prove that for
 hypergraphs $\mathcal{H}$ and $\mathcal{H}'$,
\begin{eqnarray*}
\mathcal{K}_{\mathcal{H}\boxtimes \mathcal{H}'}=\mathcal{K}_{\mathcal{H}}\times \mathcal{K}_{\mathcal{H}'}
\end{eqnarray*}
in Proposition~\ref{prop1}. Moreover, based on the embedded homology defined in \cite{h1}, we study the K\"{u}nneth formula for  the product of hypergraphs and prove that there is a natural exact sequence
\begin{equation*}
  0\rightarrow \bigoplus\limits_{p+q=n}H_{p}(\mathcal{H})\otimes H_{q}(\mathcal{H}')\rightarrow H_{n}(\mathcal{H}\boxtimes\mathcal{H}')\rightarrow
  \bigoplus\limits_{p+q=n}\mathrm{Tor}_{R}(H_{p}(\mathcal{H}),H_{q-1}(\mathcal{H}'))\rightarrow 0.
\end{equation*}
in Theorem~\ref{th-1.0}, which provides a potential tool for the data analysis of the hypergraph-modeled  networks and can be written as follows.

Finally, in Section~\ref{sec-6},  we give further discussions about  Theorem~\ref{th-1.0} and extend the K\"{u}nneth formula for digraphs proved in \cite[Section~7]{yau1} and \cite{3}.

\section{Preliminaries}\label{sec-2}
In this section, we review two important chain complex maps that model topological products by
tensor products, which are called the Eilenberg-Zilber map and the Alexander-Whitney map.  The content of this section is referred to  \cite{ann1,ann2,jpaa,rational} and \cite[Chapter~3, Section~3.B]{hatcher}.


\smallskip


First,  we review  the cross product of simplicial complexes which is essentially the Eilenberg-Zilber map  on the simplicial set level. For the details of simplicial  complexes which can be  represented by simplicial sets, the readers can refer to \cite[Chhapter~2]{wu1}.

Let $\mathcal{K}$ and $\mathcal{K}'$ be two (abstract) simplicial complexes with boundary maps $\partial$ and $\partial'$ respectively.  Take two simplices $\sigma^{(p)}\in \mathcal{K}$ and $\tau^{(q)}\in \mathcal{K}'$, where $p,q\geq 0$.  We label the vertices of $\sigma $ and $\tau$ respectively as
\begin{eqnarray*}
\sigma=\{v_0,v_1,\ldots, v_p\},  ~~~ \tau=\{w_0,w_1,\ldots,w_q\}.
\end{eqnarray*}
We regard the pairs $(i,j)$ with $0\leq i \leq p$ and $0\leq j\leq q$ as the vertices of an $p\times q$ rectangular grid in the plane.

Let $\gamma$ be a path formed by a sequence of $p$ horizontal edges and $q$ vertical edges in the grid, starting at $(0,0)$ and ending at $(p,q)$, always moving either   to the right or upward.  To such a path $\gamma$,  we associate a simplex
\begin{eqnarray*}
\eta(\gamma)=\{(v_{i_k},w_{j_k})\mid (i_k,j_k) \text{ is the } k\text{-th vertex of the path }\gamma,  0\leq k\leq m+n\}.
\end{eqnarray*}
Let $\sigma$ run over all simplices of $\mathcal{K}$, $\tau$ run over all simplices of $\mathcal{K}'$, and $\gamma$ run over all possible paths.  We obtain a simplicial complex $\mathcal{K}\times \mathcal{K}'$.  The Eilenberg-Zilber  map
\begin{eqnarray}\label{eq-3.1}
\mu: C_p(\mathcal{K};R)\otimes_R  C_q(\mathcal{K}';R)\longrightarrow C_{p+q} (\mathcal{K}\times \mathcal{K}';R)
\end{eqnarray}
is given by the rule
\begin{eqnarray}\label{eq-3.2.1}
\mu(\sigma^{(p)}\otimes \tau^{(q)})= \sum_{\gamma } (-1)^{|\gamma|}\eta(\gamma).
\end{eqnarray}
Here $|\gamma|$ is the number of squares in the grid lying below the path $\gamma$.   We use $\partial^\times$ to denote the boundary maps of $\mathcal{K}\times \mathcal{K}'$.  The boundary maps satisfy
\begin{eqnarray*}\label{eq-3.30}
\partial^\times _{p+q}(\mu(\sigma^{(p)}\otimes \tau^{(q)}))= \mu ((\partial_p\sigma^{(p)})\otimes \tau^{(q)})+ (-1)^p\mu(\sigma^{(p)}\otimes (\partial_q\tau^{(q)})).
\end{eqnarray*}
Hence $\mu$ gives a chain map by (\ref{eq-3.1}).

\smallskip

Second, we review the Alexander-Whitney map for the product of two simplices $\Delta^p \times \Delta^q$.
\begin{eqnarray*}
\nu: C_*(\Delta^p\times \Delta^q)\longrightarrow C_*(\Delta^p)\otimes C_*(\Delta^q)
\end{eqnarray*}
is  defined based on that $C_*(\Delta^p\times \Delta^q)$ is a free chain complex and $C_*(\Delta^p)\otimes C_*(\Delta^q)$ is a acyclic chain complex.
\begin{eqnarray*}
\xymatrix{
\cdots\ar[r] &C_2(\Delta^p\times \Delta^q)\ar[d]^{\nu_2}\ar[r]^{\partial_2^{\times}} & C_1(\Delta^p\times \Delta^q)\ar[d]^{\nu_1}\ar[r]^{\partial_1^{\times}}&C_0(\Delta^p\times \Delta^q)\ar[d]^{\nu_0}\\
\cdots\ar[r] &(C_*(\Delta^p)\otimes C_*(\Delta^q))_2 \ar[r]^{(\partial\otimes\partial')_2} &(C_*(\Delta^p)\otimes C_*(\Delta^q))_1 \ar[r]^{(\partial\otimes\partial')_1}&(C_*(\Delta^p)\otimes C_*(\Delta^q))_0
}
\end{eqnarray*}
Specifically, $\nu_0$ is the identity. $\nu_1$ is chosen to be any map such that $\nu_0\cdot\partial_1=\partial'_1\cdot\nu_1$ to make the diagram commute. That is,  $\nu_0\cdot\partial_1^{\times}=(\partial\otimes\partial')_1\cdot\nu_1$. After $\nu_1$ is given, $\nu_2$ is chosen to be any map that makes the diagram commute, and so on.
\smallskip

A deep-going result of $\mu$ and $\nu$ is as follows.
\begin{proposition}(cf.\cite{rational,ann1,ann2,jpaa})
\label{pr-3.111}
$\nu$ and $\mu$ are inverse chain equivalences. In fact, $\nu\circ\mu=\mathrm{id}$ and $\mu\circ\nu$ is naturally homotopic to the identity.
\qed
\end{proposition}

Finally, we give an example to illustrate the  Eilenberg-Zilber map  and  the  Alexander-Whitney map for chain complexes.
\begin{example}
Let $\Delta^{1}=\{\{0,1\}, \{0\}, \{1\}\}$ be a $1$-simplex. Then
\begin{eqnarray*}
\Delta^{1}\times \Delta^{1}&=&\{\{(0,0)\},\{(0,1)\}, \{(1,0)\}, \{(1,1)\}; \{(0,0),(0,1)\}, \{(0,0),(1,0)\},\\
&&\{(1,0),(1,1)\},\{(0,1),(1,1)\}, \{(0,0),(1,1)\};\\
&&\{(0,0),(0,1),(1,1)\}, \{(0,0),(1,0),(1,2)\}\}.
\end{eqnarray*}
$C_*(\Delta^{1})\otimes C_*(\Delta^{1})$ is the free module generated by
\begin{eqnarray*}
&&\{0\}\otimes \{0\},~ \{0\}\otimes \{1\}, ~ \{1\otimes \{0\}, ~\{1\}\otimes \{1\};\\
&&\{0\}\otimes \{0,1\}, \{0,1\}\otimes \{0, \{1\otimes \{0,1\}, \{0,1\}\otimes \{1\};\\
&&\{0,1\}\otimes \{0,1\}.
\end{eqnarray*}
By (\ref{eq-3.2.1}), $\mu: C_*(\Delta^{1})\otimes C_*(\Delta^{1})\longrightarrow C_*(\Delta^{1}\times \Delta^{1})$ is given by
\begin{eqnarray*}
&&\{0\}\otimes \{0\}\rightarrow \{(0,0)\},~~~~~ \{0\}\otimes\{1\}\rightarrow \{(0,1)\}, ~~~~~ \{1\}\otimes \{0\}\rightarrow \{(1,0)\}, ~~~~~\{1\}\otimes \{1\} \rightarrow \{(1,1)\};\\
&&\{0\}\otimes \{01\}\rightarrow \{(0,0),(0,1)\},~ \{01\}\otimes \{0 \}\rightarrow \{(0,0),(1,0)\}, ~\{1\}\otimes \{01\}\rightarrow \{(1,0),(1,1)\},\\
&&\{01\}\otimes \{1\}\rightarrow \{(0,1),(1,1)\};\{01\}\otimes \{01\}\rightarrow \{\{(0,0),(1,0),(1,1)\}-\{(0,0),(0,1),(1,1)\}\}.
\end{eqnarray*}
$\nu_0$ is given by
\begin{eqnarray*}
\{(0,0)\}\rightarrow \{0\},~~~~~\{(0,1)\} \rightarrow\{ 0\}\otimes \{1\},~~~~~\{(1,0)\}\rightarrow \{1\}\otimes \{0\},~~~~~\{(1,1)\}\rightarrow \{1\}\otimes \{1\}.
\end{eqnarray*}
To make $\nu_1$ satisfy $\nu_0\circ \partial_1^{\times}=(\partial\otimes\partial')_1\circ \nu_1$, we have that
\begin{eqnarray*}
&\{(0,0),(0,1)\}\rightarrow \{0\}\otimes \{01\},&\{(0,0),(1,0)\}\rightarrow \{01\}\otimes \{0\},\\
&\{(1,0),(1,1)\}\rightarrow \{1\}\otimes \{01\},&\{(0,1),(1,1)\} \rightarrow \{01\}\otimes \{1\}
\end{eqnarray*}
and $\nu_1(\{(0,0),(1,1)\})$ is defined to be $\{01\}\otimes 0+1\otimes \{01\}$ or to be $0\otimes \{01\}+\{01\}\otimes 1$. Then we can define $\nu_2$ based on $\nu_1\circ \partial_2^{\times}=(\partial\otimes\partial')_2\circ\nu_2$. That is, either
\begin{eqnarray*}
\nu_1(\{(0,0),(1,1)\})&=&\{0\}\otimes \{01\}+\{01\}\otimes \{1\},\\
\nu_2(\{(0,0),(1,0),(1,1)\})&=&\{01\}\otimes \{01\},\\
\nu_2(\{(0,0),(0,1),(1,1)\})&=&0
\end{eqnarray*}
or
\begin{eqnarray*}
\nu_1(\{(0,0),(1,1)\})&=&\{01\}\otimes \{0\}+\{1\}\otimes \{01\},\\
\nu_2(\{(0,0),(0,1),(1,1)\})&=&\{01\}\otimes \{01\},\\
\nu_2(\{(0,0),(1,0),(1,1)\})&=&0.
\end{eqnarray*}
Hence
\begin{eqnarray*}
\nu_2(\{(0,0),(1,0),(1,1)\}-\{(0,0),(0,1),(1,1)\})=\pm\{01\}\otimes \{01\}.
\end{eqnarray*}
\end{example}

\section{Auxiliary results for Theorem~\ref{th-1.0}}
Let $\{C_n,\partial_n\}_{n\geq 0}$ and $\{C'_n,\partial'_n\}_{n\geq 0}$ be chain complexes over $R$. For each $n\geq 0$,  we consider sub-$R$-modules $D_n\subseteq C_n$ and $D'_n\subseteq C'_n$.  Then we have graded sub-$R$-modules of the chain complexes
\begin{eqnarray*}
D=\{D_n\}_{n\geq 0}, ~~~  D'=\{D'_n\}_{n\geq 0}.
\end{eqnarray*}
The tensor product of $D$ and $D'$ is defined as
\begin{eqnarray*}
D\otimes D'=\Big\{ \bigoplus _{p+q=n,\atop p,q\geq 0} D_p\otimes D'_q\Big\}.
\end{eqnarray*}
It is direct to verify that $D\otimes D'$ is a graded sub-$R$-modules of the  chain complex $C\otimes C'$.
For simplicity,   we denote
\begin{eqnarray*}
(D\otimes D')_n=  \bigoplus _{p+q=n,\atop p,q\geq 0} D_p\otimes D'_q.
\end{eqnarray*}
For any $n\geq 0$, by direct calculation,  we have that
\begin{eqnarray}\label{eq-2.11}
(\partial\otimes \partial')_n  (D\otimes  D')_n=\sum_{p+q=n,\atop  p,q\geq 0}\partial_p D_p\otimes D'_q + D_p\otimes \partial'_q D'_q.
\end{eqnarray}

\begin{lemma}\label{le-2.2}
For any $n\geq 0$, each element in $(D\otimes D')_n\cap(\partial\otimes \partial')_n^{-1}(D\otimes D')_{n-1}$ can be written in the form
\begin{equation*}
\sum\limits_{i=1}^{m}x_{i}\otimes y_{i},\quad \mathrm{deg}(x_i)=p_i,~\mathrm{deg}(y_i)=q_i,
\end{equation*}
where  $x_i$, $y_i$ are linear combinations of elements in  $D_{p_i}$ and $D'_{q_i}$ respectively such that for each $1\leq i\leq m$,
 \begin{eqnarray*}
 (x_{i}\otimes y_{i})\in  (D\otimes D')_n\cap(\partial\otimes \partial')_n^{-1}(D\otimes D')_{n-1}.
 \end{eqnarray*}
\end{lemma}
\begin{proof}
Since $D_{p_i}$ and $D'_{q_i}$ are free $R$-modules, we can choose an arbitrary  basis $\mathcal{B}_{p_i}$ of $D_{p_i}$ and an  arbitrary  basis $\mathcal{U}_{q_i}$ of $D_{q_i}$ respectively.

Let
\begin{eqnarray}
g=r_1(\sigma_1\otimes \tau_1)+\cdots +r_l(\sigma_l\otimes \tau_l)
\end{eqnarray}
be an element in
\begin{eqnarray*}
(D\otimes D')_n\cap(\partial\otimes \partial')_n^{-1}(D\otimes D')_{n-1}
\end{eqnarray*}
where $\sigma_i\in \mathcal{B}_{p_i}$, $\tau_i\in \mathcal{U}_{q_i}$ and $r_i\in R$.

Notice that
\begin{eqnarray*}
(\partial\otimes \partial')_n(g)\in (D\otimes D')_{n-1}
\end{eqnarray*}
and
\begin{eqnarray*}
(\partial\otimes \partial')_n(\sigma_1\otimes \tau_1)=\partial \sigma_1\otimes \tau_1+(-1)^{\mathrm{deg}(\sigma_1)}\sigma_1\otimes \partial'\tau_1.
\end{eqnarray*}

We consider the following  two cases.

{\sc Case~1}. $\partial \sigma_1\in D_{p_i-1}$ and $\partial'\tau_1\in D'_{q_i-1}$.  Then
\begin{eqnarray}
\sigma_1\otimes \tau_1\in (D\otimes D')_n\cap(\partial\otimes \partial')_n^{-1}(D\otimes D')_{n-1}.
\end{eqnarray}
Set $x_1=r_1\sigma_1$ and $y_1=\tau_1$.

{\sc Case~2}. $\partial \sigma_1\notin D_{p_1-1}$ or $\partial'\tau_1\notin D'_{q_1-1}$. Without loss of  generality, we can assume that $\partial \sigma_1\notin D_{p_1-1}$. Then there exists $0\leq k_1\leq p_1$ such that $d_{k_1}(\sigma_1)\notin D_{p_1-1}$.  Since $(\partial\otimes \partial')_n(g)\in (D\otimes D')_{n-1}$, then 
\begin{eqnarray*}
\sum\limits_{\{(j, k_j)\mid d_{k_j}(\sigma_j)=d_{k_1}(\sigma_1)\}}(-1)^{k_j}r_j=0,
\end{eqnarray*}
with $\tau_1=\tau_j$, $p_1=p_j$ and $q_1=q_j$. Hence we get a linear combination $x_1=(r_1\sigma_1+r_j\sigma_j+\cdots)$ of finite elements in $\mathcal{B}_{p_1}$ and $y_1=\tau_1$.

Combing Case 1 and Case 2, the lemma follows.
\end{proof}

We give an example to illustrate the form of element in $(D\otimes D')_n\cap(\partial\otimes \partial')_n^{-1}(D\otimes D')_{n-1}$.
\begin{example}\label{ex-01}
Let
\begin{eqnarray*}
\mathcal{H}&=&\{\{v_1\}, \{v_3\}, \{v_1,v_2\}, \{v_2,v_3\}, \{v_1,v_3\}\},\\
\quad\mathcal{H}'&=&\{\{w_2\}, \{w_3\}, \{w_1,w_2\}, \{w_2,w_3\}, \{w_1,w_3\}\}
\end{eqnarray*}
be two hypergraphs. Let $D, D'$ be the free $\mathbb{Z}$-modules generated by $\mathcal{H}$ and $\mathcal{H}'$ respectively.
Let
\begin{eqnarray*}
&&g=\{v_1,v_2\}\otimes \{w_2,w_3\}+\{v_1,v_3\}\otimes\{w_1,w_3\}+\{v_2,v_3\}\otimes\{w_2,w_3\}\\
&&~~~~~~~~~-\{v_1,v_3\}\otimes\{w_1,w_2\}+\{v_1,v_3\}\otimes\{w_2,w_3\}.
\end{eqnarray*}
It can be directly verified that
\begin{eqnarray*}
g\in (D\otimes D')_2\cap(\partial\otimes \partial')_2^{-1}(D\otimes D')_1,~\{v_1,v_3\}\otimes\{w_2,w_3\}\in (D\otimes D')_2\cap(\partial\otimes \partial')_2^{-1}(D\otimes D')_1
\end{eqnarray*}
and none of
\begin{eqnarray*}
\{v_1,v_2\}\otimes \{w_2,w_3\},& \{v_1,v_3\}\otimes\{w_1,w_3\},\\
\{v_2,v_3\}\otimes\{w_2,w_3\},&\{v_1,v_3\}\otimes\{w_1,w_2\}
\end{eqnarray*}
is in $(D\otimes D')_2\cap(\partial\otimes \partial')_2^{-1}(D\otimes D')_1$. But we can express $g$ as
\begin{eqnarray*}
g=(\{v_1,v_2\}+\{v_2,v_3\})\otimes \{w_2,w_3\}+\{v_1,v_3\}\otimes (\{w_1,w_3\}-\{w_1,w_2\})+\{v_1,v_3\}\otimes\{w_2,w_3\},
\end{eqnarray*}
in which each term
\begin{eqnarray*}
(\{v_1v_2\}+\{v_2v_3\}])\otimes \{w_2w_3\}, ~~ \{v_1v_3\}\otimes (\{w_1w_3\}-\{w_1w_2\},~~\{v_1,v_3\}\otimes\{w_2,w_3\}
\end{eqnarray*}
is  in $(D\otimes D')_2\cap(\partial\otimes \partial')_2^{-1}(D\otimes D')_1$.
\end{example}
\smallskip

\begin{proposition}\label{pr-2.1}
Let $C,C'$ be chain complexes of finitely generated free $R$-modules. Let $D,D'$ be graded sub-$R$-modules of $C,C'$, respectively. Then for any $n\geq 0$, we have
\begin{eqnarray*}
\text{Inf}_n(D\otimes D',C\otimes C')=(\text{Inf}_*(D,C)\otimes \text{Inf}_*(D',C'))_n.
\end{eqnarray*}
\end{proposition}
\begin{proof}

By substituting $D$ and $D'$ with $D\cap \partial^{-1}D$  and $D'\cap \partial'^{-1}D'$ in (\ref{eq-2.11}) respectively, we have  that
\begin{eqnarray*}\label{eq-6.2}
&&(\partial\otimes \partial')_n ((D\cap \partial^{-1}D)\otimes (D'\cap \partial'^{-1}D'))_n \\
&&=\sum \Big(\partial_p(D\cap \partial^{-1}D)_p\otimes (D'\cap\partial'^{-1}D')_q+(D\cap \partial^{-1}D)_p\otimes\partial'_q (D'\cap\partial'^{-1}D')_q\Big)\\
&&\subseteq (D\otimes D')_{n-1}.
\end{eqnarray*}
Moreover,
\begin{eqnarray*}\label{eq-6.1}
((D\cap \partial^{-1}D)\otimes (D'\cap \partial'^{-1}D'))_n \subseteq (D\otimes D')_{n}.
\end{eqnarray*}
Hence
\begin{eqnarray*}
((D\cap \partial^{-1}D)\otimes (D'\cap \partial'^{-1}D'))_n  \subseteq (D\otimes D')_{n}\cap (\partial\otimes \partial')_n (D\otimes D')_{n-1},
\end{eqnarray*}
which implies that
\begin{equation*}
(\text{Inf}_*(D,C)\otimes \text{Inf}_*(D',C'))_n\subseteq\text{Inf}_n(D\otimes D',C\otimes C').
\end{equation*}

On the other hand, for each factor
\begin{eqnarray*}
x\otimes y\in  (D\otimes D')_n\cap(\partial \otimes \partial')_{n}^{-1}(D\otimes D')_{n-1},~\mathrm{deg}(x)=p,~\mathrm{deg}(y)=q, ~p+q=n,
\end{eqnarray*}
we  have that
\begin{equation*}
  (\partial \otimes \partial')_n (x\otimes y)=((\partial_p x)\otimes y+(-1)^p x\otimes (\partial'_q y))\in (D\otimes D')_{n-1}.
\end{equation*}
Then
\begin{equation*}\label{eq-2.2.2}
  (\partial_p x)\otimes y\in (D\otimes D')_{n-1},~ x\otimes (\partial'_q y)\in (D\otimes D')_{n-1}.
\end{equation*}
Hence
\begin{eqnarray*}
 x \in D_p\cap \partial_p^{-1} D_{p-1}, \quad y \in D'_q \cap {\partial'}_q^{-1} D'_{q-1}
 \end{eqnarray*}
and
\begin{equation*}
  (x\otimes y)\in ((D\cap \partial^{-1}D)\otimes (D'\cap \partial'^{-1}D'))_n.
\end{equation*}

Therefore,  by Lemma~\ref{le-2.2},
\begin{eqnarray*}
(D\otimes D')_n\cap(\partial \otimes \partial')_n^{-1}(D\otimes D')_{n-1}\subseteq ((D\cap \partial^{-1}D)\otimes (D'\cap \partial'^{-1}D'))_n.
\end{eqnarray*}
which implies that
\begin{equation*}
\text{Inf}_n(D\otimes D',C\otimes C')\subseteq(\text{Inf}_*(D,C)\otimes \text{Inf}_*(D',C'))_n.
\end{equation*}

The proposition is proved.
\bigskip
\end{proof}

\section{A K\"{u}nneth Formula for Product of Hypergraphs}\label{sec-4}
In this section, we define the product of hypergraphs and prove a K\"{u}nneth formula for product of hypergraphs.

\smallskip

\subsection{Product of Hypergraphs}
 We  define the product of hypergraphs and  give the connection between the product of hypergarphs and the Cartesian product of
 their associated simplicial complexes, in Proposition~\ref{prop1}.
\begin{definition}
Let $\mathcal{H},\mathcal{H}'$ be hypergraphs. Define the \emph{product of hypergraphs} $\mathcal{H}\boxtimes\mathcal{H}'$ as a hypergraph whose elements are in the form of
\begin{equation*}
  \omega=\{(v_{\alpha(0)},w_{\beta(0)}),(v_{\alpha(1)},w_{\beta(1)})\cdots,(v_{\alpha(p+q)},w_{\beta(p+q)})\}
\end{equation*}
for all $\sigma=\{v_{0},\cdots,v_{p}\}\in \mathcal{H},\tau=\{w_{0},\cdots,w_{q}\}\in \mathcal{H}'$, where for each $i$ either
\begin{equation*}
  \left\{
    \begin{array}{ll}
      \alpha(i+1)=\alpha(i)  \\
      \beta(i+1)=\beta(i)+1
    \end{array}
  \right.\quad  \textmd{or~} \quad \left\{
    \begin{array}{ll}
      \alpha(i+1)=\alpha(i)+1  \\
      \beta(i+1)=\beta(i)
    \end{array}
  \right.
\end{equation*}
\end{definition}

Note that if  $\mathcal{H},\mathcal{H}'$ are both simplicial complexes, then  the product $\mathcal{H}\boxtimes \mathcal{H}'$ is not always a simplicial complex. But its associated simplicial complex
coincides with the Cartesian product of simplicial complexes. That is, $\mathcal{K}_{\mathcal{H}\boxtimes \mathcal{H}'}=\mathcal{H}\times \mathcal{H}'$ when $\mathcal{H},\mathcal{H}'$  are both simplicial complexes.  More generally, we have that
\begin{proposition}\label{prop1}
Let $\mathcal{H}$, $\mathcal{H}'$ be hypergraphs. Then $\mathcal{K}_{\mathcal{H}\boxtimes \mathcal{H}'}=\mathcal{K}_{\mathcal{H}}\times \mathcal{K}_{\mathcal{H}'}$.
\end{proposition}
\begin{proof}
Note that $\mathcal{H}\boxtimes \mathcal{H}'\subseteq \mathcal{K}_{\mathcal{H}}\times \mathcal{K}_{\mathcal{H}'}$. Since $\mathcal{K}_{\mathcal{H}}\times \mathcal{K}_{\mathcal{H}'}$ is a simplicial complex, by definition, we have $\mathcal{K}_{\mathcal{H}\boxtimes \mathcal{H}'}\subseteq \mathcal{K}_{\mathcal{H}}\times \mathcal{K}_{\mathcal{H}'}$.

For each $\omega\in \mathcal{K}_{\mathcal{H}}\boxtimes \mathcal{K}_{\mathcal{H}'}$, we assume that
\begin{equation*}
  \omega=\{(v_{i_{\alpha(0)}},w_{j_{\beta(0)}}),(v_{i_{\alpha(1)}},w_{j_{\beta(1)}}),\cdots,(v_{i_{\alpha(p+q)}},w_{j_{\beta(p+q)}})\},
\end{equation*}
where $\sigma=\{v_{i_{0}},\cdots,v_{i_{p}}\}\in \mathcal{K}_{\mathcal{H}},\tau=\{w_{j_{0}},\cdots,w_{j_{q}}\}\in \mathcal{K}_{\mathcal{H}'}$. Thus there exist two elements $\tilde{\sigma}\in \mathcal{H}$ and $\tilde{\tau}\in  \mathcal{H}'$ such that
\begin{equation*}
  \sigma\subseteq\tilde{ \sigma},\quad \tau\subseteq \tilde{ \tau}.
\end{equation*}
We may write
\begin{equation*}
  \tilde{ \sigma}=\{v_{0},\cdots,v_{m}\},\quad \tilde{\tau}=\{w_{0},\cdots,w_{n}\}
\end{equation*}
and choose $i_{0},\cdots,i_{p}$ and $j_{0},\cdots,j_{q}$ corresponding to the order of elements in $\tilde{\sigma}$ and $\tilde{\tau}$. For $\tilde{\sigma}\in \mathcal{H}$ and $\tilde{\tau}\in  \mathcal{H}'$, we obtain an element  $ \tilde{\omega}\in\mathcal{H}\boxtimes \mathcal{H}'$ as follows.
\begin{equation*}
  \tilde{\omega}= \left\{\begin{array}{cccc}
(v_{0},w_{0}), & (v_{1},w_{0}), & \cdots & (v_{i_{\alpha(0)}},w_{0}) ,\\
(v_{i_{\alpha(0)}},w_{1}), & (v_{i_{\alpha(0)}},w_{2}), & \cdots & (v_{i_{\alpha(0)}},w_{j_{\beta(0)}}) ,\\
(v_{i_{\alpha(0)}+1},w_{j_{\beta(0)}}) , & (v_{i_{\alpha(0)}+2},w_{j_{\beta(0)}}),  & \cdots & (v_{i_{\alpha(1)}},w_{j_{\beta(0)}}) , \\
 (v_{i_{\alpha(1)}},w_{j_{\beta(0)}+1}), &  (v_{i_{\alpha(1)}},w_{j_{\beta(0)}+2}), & \cdots &  (v_{i_{\alpha(1)}},w_{j_{\beta(1)}}), \\
 \vdots&\vdots&\vdots&\vdots\\
(v_{i_{\alpha(p+q)}},w_{j_{\beta(p+q-1)}+1}), &(v_{i_{\alpha(p+q)}},w_{j_{\beta(p+q-1)}+2}), & \cdots & (v_{i_{\alpha(p+q)}},w_{j_{\beta(p+q)}}) ,\\
(v_{i_{\alpha(p+q)}+1},w_{j_{\beta(p+q)}}), &(v_{i_{\alpha(p+q)}+2},w_{j_{\beta(p+q)}}), & \cdots & (v_{m},w_{j_{\beta(p+q)}}), \\
 (v_{m},w_{j_{\beta(p+q)}+1}), &  (v_{m},w_{j_{\beta(p+q)}+2}), & \cdots & (v_{m},w_{n}) \\
                 \end{array}\right\}
\end{equation*}
It follows that $\omega\subseteq \tilde{\omega}$. Thus we have $\omega\in \mathcal{K}_{\mathcal{H}\boxtimes \mathcal{H}'}$. It shows that $\mathcal{K}_{\mathcal{H}}\boxtimes \mathcal{K}_{\mathcal{H}'}\subseteq \mathcal{K}_{\mathcal{H}\boxtimes \mathcal{H}'}$. Thus we have $\mathcal{K}_{\mathcal{H}}\times \mathcal{K}_{\mathcal{H}'}= \mathcal{K}_{\mathcal{K}_{\mathcal{H}}\boxtimes \mathcal{K}_{\mathcal{H}'}}\subseteq \mathcal{K}_{\mathcal{H}\boxtimes \mathcal{H}'}$. The desired result follows.
\end{proof}

Furthermore, to illustrate the relationship between the  product $\boxtimes$ of hypergraphs and the Cartesian product $\times$ of simplicial complexes more vividly, we give the following example.
\begin{example}
Let $\mathcal{H}=\{\{v_0\}, \{v_0,v_1\}\}$, $\mathcal{H'}=\{\{w_1\}, \{w_0,w_1\}\}$. Then
\begin{eqnarray*}
\mathcal{K}_{\mathcal{H}}&=&\{\{v_0\}, \{v_1\}, \{v_0,v_1\}\},\\
\mathcal{K}_{\mathcal{H'}}&=&\{\{w_0\}, \{w_1\}, \{w_0,w_1\}\}\\
\mathcal{H}\boxtimes \mathcal{H}'&=&\{\{(v_0,w_1)\}, \{(v_0,w_0), (v_0,w_1)\}, \{(v_0,w_1), (v_1,w_1)\},\\ &&\{(v_0,w_0),(v_1,w_0),(v_1,w_1)\}, \{(v_0,w_0),(v_0,w_1),(v_1,w_1)\}\},\\
\mathcal{K}_{\mathcal{H}}\boxtimes \mathcal{K}_{\mathcal{H}'}&=&\{\{(v_0,w_0)\}, \{(v_1,w_0)\}, \{(v_0,w_1)\}, \{(v_1,w_1)\},\\
&&\{(v_0,w_0),(v_1,w_0)\}, \{(v_0,w_1),(v_1,w_1)\}, \{(v_0,w_0),(v_0,w_1)\}, \{(v_1,w_0),(v_1,w_1)\},\\
&&\{(v_0,w_0),(v_1,w_0),(v_1,w_1)\}, \{(v_0,w_0),(v_0,w_1),(v_1,w_1)\}\},\\
\mathcal{K}_{\mathcal{H}}\times\mathcal{K}_{\mathcal{H}'}&=&\mathcal{K}_{\mathcal{H}} \boxtimes \mathcal{K}_{\mathcal{H}'}\cup \{(v_0,w_0),(v_1,w_1)\}.
\end{eqnarray*}
\end{example}

\subsection{Proof of K\"{u}nneth Formula for Hypergraphs}
Before giving the proof of Theorem~\ref{th-1.0}, we give some properties of the restrictions of  $\mu$ and $\nu$ on sub-chain complexes firstly. For convenience, we denote
\begin{eqnarray*}
\mathrm{Inf}_{n}(R(\mathcal{H})_{\ast}\otimes R(\mathcal{H}')_{\ast},C_{\ast}(\mathcal{K}_{\mathcal{H}};R)\otimes C_{\ast}(\mathcal{K}_{\mathcal{H}'};R))
\end{eqnarray*}
as
\begin{eqnarray*}
\mathrm{Inf}_{n}(R(\mathcal{H})_{\ast}\otimes R(\mathcal{H}')_{\ast})\label{eq-1a}
\end{eqnarray*}
and denote
\begin{eqnarray*}
\mathrm{Inf}_{n}(R(\mathcal{H}\boxtimes \mathcal{H}')_{\ast},C_{\ast}(\mathcal{K}_{\mathcal{H}}\times \mathcal{K}_{\mathcal{H}'};R))
\end{eqnarray*}
as
\begin{eqnarray*}
\mathrm{Inf}_{n}(R(\mathcal{H}\boxtimes \mathcal{H}')_{\ast})\label{eq-8.1}
\end{eqnarray*}
respectively in this subsection.

\smallskip
\begin{lemma}\label{le-6}
The restriction of chain map $\mu$:
\begin{eqnarray}\label{eq-61}
\mu\mid_{\mathrm{Inf}_{n}(R(\mathcal{H})_{\ast}\otimes R(\mathcal{H}')_{\ast})}
\end{eqnarray}
 and that of chain map $\nu$:
 \begin{eqnarray}\label{eq-62}
\nu\mid_{\mathrm{Inf}_{n}(R(\mathcal{H}\boxtimes \mathcal{H}')_{\ast})}
\end{eqnarray}
are still chain maps. Moreover,
\begin{equation}\label{eq-63}
  \nu\mid_{\mathrm{Inf}_{n}(R(\mathcal{H}\boxtimes \mathcal{H}')_{\ast})}\circ \mu|_{\mathrm{Inf}_{n}(R(\mathcal{H})_{\ast}\otimes R(\mathcal{H}')_{\ast})}=\mathrm{id},
\end{equation}
\end{lemma}

\begin{proof}
By the definition of Eilenberg-Zilber map, we have that for $\sigma\in \mathcal{H}$ and $\tau\in \mathcal{H}'$,
\begin{equation*}
 \mu(\sigma\otimes \tau)=\sum\limits_{\omega\in \mathcal{H}\boxtimes \mathcal{H}'}(-1)^{|\omega|}\omega\in R(\mathcal{H}\boxtimes \mathcal{H}')_{\ast},\quad 
\end{equation*}
where  $|\omega|$ is determined by $\omega$.
Hence
$$\mu(R(\mathcal{H})_{\ast}\otimes R(\mathcal{H}')_{\ast})\subseteq R(\mathcal{H}\boxtimes \mathcal{H}')_{\ast}.$$
On the other hand, by straightforward calculation, we have that
\begin{equation*}
  \nu(R(\mathcal{H}\boxtimes \mathcal{H}')_{\ast})\subseteq R(\mathcal{H})_{\ast}\otimes R(\mathcal{H}')_{\ast}.
\end{equation*}
Thus
\begin{equation*}
\mu(\mathrm{Inf}_{n}(R(\mathcal{H})_{\ast}\otimes R(\mathcal{H}')_{\ast})\subseteq \mathrm{Inf}_{n}(R(\mathcal{H}\boxtimes \mathcal{H}')_{\ast})
\end{equation*}
and
\begin{equation*}
\nu(\mathrm{Inf}_{n}(R(\mathcal{H}\boxtimes \mathcal{H}')_{\ast}\subseteq\mathrm{Inf}_{n}(R(\mathcal{H})_{\ast}\otimes R(\mathcal{H}')_{\ast})
\end{equation*}
Hence  (\ref{eq-61}) and (\ref{eq-62}) are well defined.

Next, we will show that (\ref{eq-61}) and (\ref{eq-62}) are chain maps. Consider the following diagram
\begin{eqnarray*}\label{eq-63}
\xymatrix{
\cdots\ar[r]&\mathrm{Inf}_{n}(R(\mathcal{H})_{\ast}\otimes R(\mathcal{H}')_{\ast}) \ar@/_/[d]_{\mu_n}
\ar[r]^{(\partial\otimes \partial')_n}
 & \mathrm{Inf}_{n-1}(R(\mathcal{H})_{\ast}\otimes R(\mathcal{H}')_{\ast})\ar@/_/[d]_{\mu_{n-1}}
 \ar[r]&\cdots\\
\cdots\ar[r] &\mathrm{Inf}_{n}(R(\mathcal{H}\boxtimes \mathcal{H}')_{\ast})\ar@/_/[u]_{\nu_n}\ar[r]^{\partial^\times_n}
& \mathrm{Inf}_{n-1}(R(\mathcal{H}\boxtimes \mathcal{H}')_{\ast})
\ar@/_/[u]_{\nu_{n-1}}\ar[r]&\cdots
}
\end{eqnarray*}

Step 1. We claim that (\ref{eq-61}) is a chain map.

By Lemma~\ref{le-2.2}, each element in $\mathrm{Inf}_{n}(R(\mathcal{H})_{\ast}\otimes R(\mathcal{H}')_{\ast})$ can be written as
\begin{equation*}
g=\sum\limits_{i=1}^{m}x_{i}\otimes y_{i},\quad \mathrm{deg}(x_i)=p_i,~\mathrm{deg}(y_i)=q_i, ~p_i+q_i=n,
\end{equation*}
where  $x_i$, $y_i$ are linear combinations of hyperedges in  $\mathcal{H}$ and $\mathcal{H}'$ respectively such that for each $1\leq i\leq m$,
\begin{eqnarray*}
(x_{i}\otimes y_{i})\in \mathrm{Inf}_{n}(R(\mathcal{H})_{\ast}\otimes R(\mathcal{H}')_{\ast})\label{eq-1a}.
\end{eqnarray*}

Hence
\begin{eqnarray*}
(\partial\otimes\partial')_n(x_i\otimes y_i)=\partial_{p_i} x_i\otimes y_i+(-1)^{p_i} x_i\otimes \partial'_{q_i} y_i
\end{eqnarray*}
is an element in $(R(\mathcal{H})_{\ast}\otimes R(\mathcal{H}')_{\ast})_{n-1}$ and
\begin{eqnarray*}
(\partial\otimes\partial')_n(x_i\otimes y_i)\in \mathrm{Inf}_{n-1}(R(\mathcal{H})_{\ast}\otimes R(\mathcal{H}')_{\ast}).
\end{eqnarray*}

On the other hand, we have that
\begin{eqnarray*}
\mu_n(x_i\otimes y_i)=x_i\boxtimes y_i\in R(\mathcal{H}\boxtimes \mathcal{H}')_{\ast}\subseteq C_{\ast}(\mathcal{K}_{\mathcal{H}}\times \mathcal{K}_{\mathcal{H}'};R).
\end{eqnarray*}
and
\begin{eqnarray*}
\partial_n^{\times}(x_i\boxtimes y_i)&=&\big((\partial_{p_i} x_i)\boxtimes y_i+(-1)^{p_i}x_i\boxtimes (\partial'_{q_i}y_i)\big)\\
&&\in R(\mathcal{H}\boxtimes\mathcal{H}')_{n-1}.
\end{eqnarray*}

Hence
\begin{eqnarray*}
(\mu_{n-1}\circ(\partial\otimes\partial')_n)(x_i\otimes y_i)&=&\mu_{n-1}\big((\partial_{p_i} x_i)\otimes y_i+(-1)^{p_i} x_i\otimes (\partial'_{q_i} y_i)\big)\\
&=&(\partial_{p_i} x_i)\boxtimes y_i+(-1)^{p_i}x_i\boxtimes (\partial'_{q_i} y_i)\\
&=&(\partial_n^{\times}\circ \mu_n)(x_i\otimes y_i),
\end{eqnarray*}
which implies  that  $(\ref{eq-61})$ is a chain map. We  have obtained  the  proof of assertion in Step 1.

Step 2. We claim that $(\ref{eq-62})$ is a chain map.

Let $g=\sum r_i\omega_i$ be an arbitrary element  in $\mathrm{Inf}_{n}(R(\mathcal{H}\boxtimes \mathcal{H}')_{\ast})$ where
\begin{equation*}
  \omega_i=\{(v_{\alpha(0)}^{(i)},w_{\beta(0)}^{(i)}),(v_{\alpha(1)}^{(i)},w_{\beta(1)}^{(i)})\cdots,(v_{\alpha(p_i+q_i)}^{(i)},w_{\beta(p_i+q_i)}^{(i)})\}\in (\mathcal{H}\boxtimes \mathcal{H}')_n
\end{equation*}
for some
\begin{eqnarray*}
\sigma_i=\{v_{\alpha(0)}^{(i)},v_{\alpha(1)}^{(i)},\cdots,v_{\alpha(p_i)}^{(i)}\}\in \mathcal{H}_p
\end{eqnarray*}
 and
 \begin{eqnarray*}
 \tau_i=\{w_{\beta(0)}^{(i)},w_{\beta(1)}^{(i)},\cdots,w_{\beta(q_i)}^{(i)}\}\in\mathcal{H}'_q
 \end{eqnarray*}
 with $p_i+q_i=n$, $r_i\in R$. There are two cases.

{\sc Case 1}. $p_i\geq 1$ and  $q_i\geq1$. Then there must exist some $1\leq k\leq n$  such that either
\begin{equation*}\label{eq-84}
  \left\{
    \begin{array}{ll}
       \alpha(k+1)=\alpha(k)+1=\alpha(k-1)+1  \\
       \beta(k+1)=\beta(k)=\beta(k-1)+1
    \end{array}
  \right.
  \end{equation*}
or
\begin{equation*}\label{eq-90}
  \left\{
    \begin{array}{ll}
       \alpha(k+1)=\alpha(k)=\alpha(k-1)+1  \\
       \beta(k+1)=\beta(k)+1=\beta(k-1)+1.
    \end{array}
  \right.
\end{equation*}
Without loss of generality, suppose
\begin{equation*}
  \omega_i=\{\cdots,(v_{\alpha(k-1)}^{(i)},w_{\beta(k-1)}^{(i)}), (v_{\alpha(k)}^{(i)},w_{\beta(k-1)+1}^{(i)}), (v_{\alpha(k-1)+1}^{(i)},w_{\beta(k-1)+1}^{(i)}),\cdots\}.
\end{equation*}

Since $\partial_n^{\times} g\in R(\mathcal{H}\boxtimes\mathcal{H}')_{n-1}$ and $d_k(\omega_i)\notin (\mathcal{H}\boxtimes\mathcal{H}')_{n-1}$, by an  argument similar to \cite[P. 66]{yau1}, it follows that
\begin{equation*}
\tilde\omega_i=\{\cdots,(v_{\alpha(k-1)}^{(i)},w_{\beta(k-1)}^{(i)}), (v_{\alpha(k-1)+1}^{(i)},w_{\beta(k-1)}^{(i)}), (v_{\alpha(k-1)+1}^{(i)},w_{\beta(k-1)+1}^{(i)}),\cdots\}
\end{equation*}
is also a factor of  $g$ and all elements of $\sigma_i\boxtimes\tau_i$ are factors of $g$.

For any $0\leq j\leq n$ such that either $\alpha(j+1)\not=\alpha(j-1)+1 $ or $\beta(j+1)\not=\beta(j-1)+1$, $d_j(\omega)$ must be in one of the following forms
\begin{eqnarray*}
&&\{\cdots,(v_{\alpha(j-1)}^{(i)},w_{\beta(j-1)}^{(i)}), \widehat{(v_{\alpha(j-1)}^{(i)},w_{\beta(j)}^{(i)})}, (v_{\alpha(j-1)}^{(i)},w_{\beta(j-1)+2}^{(i)}),\cdots\},j>0,\\
&&\{\cdots,(v_{\alpha(j-1)}^{(i)},w_{\beta(j-1)}^{(i)}), \widehat{(v_{\alpha(j)}^{(i)},w_{\beta(j-1)}^{(i)})}, (v_{\alpha(j-1)+2}^{(i)},w_{\beta(j-1)}^{(i)}),\cdots\},j>0,\\
&&\{\widehat{(v_{\alpha(0)}^{(i)},w_{\beta(0)}^{(i)})},(v_{\alpha(0)}^{(i)},w_{\beta(0)+1}^{(i)}),\cdots\}, j=0,\\
&&\{\widehat{(v_{\alpha(0)}^{(i)},w_{\beta(0)}^{(i)})},(v_{\alpha(0)+1}^{(i)},w_{\beta(0)}^{(i)}),\cdots\}, j=0,
\end{eqnarray*}
which are factors of either
\begin{eqnarray}\label{eq-001}
\sigma_i\boxtimes \{w_{\beta(0)}^{(i)},\cdots,\widehat{w_{\beta(j)}^{(i)}},\cdots,w_{\beta(q_i)}^{(i)}\}
\end{eqnarray}
 or
\begin{eqnarray}\label{eq-002}
\{v_{\alpha(0)}^{(i)},\cdots, \widehat{v_{\alpha(j)}^{(i)}},\cdots,v_{\alpha(p_i)}^{(i)}\}\boxtimes \tau_i.
\end{eqnarray}

Since $\partial_n^{\times} g\in R(\mathcal{H}\boxtimes\mathcal{H}')_{n-1}$, it follows that  all factors in (\ref{eq-001}) and (\ref{eq-002}) which are not in $R(\mathcal{H})_{p_i}\boxtimes R(\mathcal{H}')_{q_{i}-1}$ or $R(\mathcal{H})_{p_{i}-1}\boxtimes R(\mathcal{H}')_{q_i}$ must be cancelled out  within $\partial_n^{\times} g$. Hence $g$ can be written as the sum of $x_i\boxtimes y_i$  where $x_i$ and $y_i$ are linear combinations of hyperedges in $\mathcal{H}_{p_i}$ and $\mathcal{H}'_{q_i}$ respectively such that $\partial x_i\in R(\mathcal{H})_{p_{i}-1}$ and  $\partial' y_i\in R(\mathcal{H}')_{{q_i}-1}$.


{\sc Case 2}. $p_i=0$ or $q_i=0$. Without loss of generality, suppose $q_i=0$. Then $p_i=n$. Let
\begin{eqnarray*}
\omega_i& =&\{(v_{\alpha_0}^{(i)},w_{\beta_0}^{(i)}),\cdots,(v_{\alpha_n}^{(i)},w_{\beta_0}^{(i)})\}\\
& =&\{v_{\alpha_0}^{(i)},\cdots,v_{\alpha_n}^{(i)}\}\boxtimes \{w_{\beta_0}^{(i)}\}
\end{eqnarray*}
 be an element in $(\mathcal{H}\boxtimes \mathcal{H}')_n$. Then $\sigma_i=\{v_{\alpha_0}^{(i)},\cdots,v_{\alpha_n}^{(i)}\}\in \mathcal{H}$ and $\tau_i=\{w_{\beta_0}^{(i)}\}\in \mathcal{H}'$.

Since $\partial_n^{\times} g\in R(\mathcal{H}\boxtimes\mathcal{H}')_{n-1}$, it follows that  for any $d_{k_i}(\sigma_i)\notin \mathcal{H}_{n-1}$ ($0\leq k_i\leq n$), there exists  some other $n$-hyperedges $\sigma_j$ in $\mathcal{H}$ such that
\begin{eqnarray*}
\sum\limits_{\{(j,k_j)\mid d_{k_j}(\sigma_j)=d_{k_i}(\sigma_i)\}}(-1)^{k_j}=0
\end{eqnarray*}
in $\partial_n^{\times} g$. Hence  $g$ can be written as
\begin{eqnarray*}
\sum x_i\boxtimes \{w_{\beta_0}^{(i)}\}
\end{eqnarray*}
where $x_i$ is  a linear combination of $n$-hyperedges of $\mathcal{H}$ such that  $\partial x_i\in R(\mathcal{H})_{n-1}$ and $\{w_{\beta_0}^{(i)}\}$ are $0$-hyperedges in $\mathcal{H}'$. 

Combining Case 1 and Case 2, we have that any element  in $\mathrm{Inf}_{n}(R(\mathcal{H}\boxtimes \mathcal{H}')_{\ast})$ can be written as $g=\sum\limits_i^m x_i\boxtimes y_i$  where $x_i$ and $y_i$ are linear combinations of hyperedges of $\mathcal{H}$ and $\mathcal{H}'$ respectively such that $\partial x_i\in R(\mathcal{H})_{p_{i}-1}$ and  $\partial' y_i\in R(\mathcal{H}')_{q_{i}-1}$.

Then \begin{eqnarray*}
x_i\in \mathrm{Inf}_{p_i}(R(\mathcal{H})_{\ast},C_{\ast}(\mathcal{K}_{\mathcal{H}};R))
\end{eqnarray*}
 and
 \begin{eqnarray*}
y_i\in \mathrm{Inf}_{q_i}(R(\mathcal{H}')_{\ast},C_{\ast}(\mathcal{K}_{\mathcal{H}'};R)).
\end{eqnarray*}
Hence
\begin{eqnarray*}
x_i\otimes y_i\in \mathrm{Inf}_{p_i}(R(\mathcal{H})_{\ast},C_{\ast}(\mathcal{K}_{\mathcal{H}};R))\otimes \mathrm{Inf}_{q_i}(R(\mathcal{H}')_{\ast},C_{\ast}(\mathcal{K}_{\mathcal{H}'};R)).
\end{eqnarray*}

By Proposition~\ref{pr-2.1}, $x_i\otimes y_i$ is an element in $\mathrm{Inf}_{n}(R(\mathcal{H})_{\ast}\otimes R(\mathcal{H}')_{\ast})$. Then
\begin{eqnarray*}
(\nu_{n-1}\circ\partial_n^{\times})(x_i\boxtimes y_i)&=&\nu_{n-1}((\partial_{p_i} x)\boxtimes y+(-1)^{p_i} x\boxtimes (\partial_{q_i} y) )\\
&=&(\partial_{p_i} x)\otimes y+(-1)^{p_i} x\otimes (\partial_{q_i} y) \\
&=&((\partial\otimes\partial')_n\circ \nu_n)(x_i\otimes y_i)
\end{eqnarray*}
which implies  $(\ref{eq-62})$ is a chain map. We have obtained the proof of assertion in Step 2.

Finally, by the action of (\ref{eq-61}) and (\ref{eq-62}) in Step 1 and Step 2 respectively, we have (\ref{eq-63}).

The lemma is proved.
\end{proof}
\smallskip
In the following example, we illustrate the commutativity of $\mu$ and boundary operator.
\begin{example}
Let $\mathcal{H}$ and $\mathcal{H}'$ be hypergraphs given in Example~\ref{ex-01}. Then
\begin{eqnarray*}
g_1=(\{v_1,v_2\}+\{v_2,v_3\})\otimes \{w_2,w_3\}
\end{eqnarray*}
is an element in $\mathrm{Inf}_{2}(R(\mathcal{H}\boxtimes \mathcal{H}')_{\ast},C_{\ast}(\mathcal{K}_{\mathcal{H}}\times \mathcal{K}_{\mathcal{H}'};R))$. Hence,
\begin{eqnarray*}
(\partial\otimes\partial')_2(g_1)&=&(\{v_3\}-\{v_1\})\otimes\{w_2,w_3\}+(-1)(\{v_1,v_2\}+\{v_2,v_3\})\otimes (\{w_3\}-\{w_2\})\\
(\mu_1\circ(\partial\otimes\partial')_2)(g_1)&=&(\{v_3\}-\{v_1\})\boxtimes\{w_2,w_3\}+(-1)(\{v_1,v_2\}+\{v_2,v_3\})\boxtimes (\{w_3\}-\{w_2\})\\
&=&\{(v_3,w_2),(v_3,w_3)\}-\{(v_1,w_2),(v_1,w_3)\}-\{(v_1,w_3),(v_2,w_3)\}\\
&&+\{(v_1,w_2),(v_2,w_2)\}-\{(v_2,w_3),(v_3,w_3)\}+\{(v_2,w_2),(v_3,w_2)\}\\
\mu_2(g_1)&=&(\{v_1,v_2\}+\{v_2,v_3\})\boxtimes \{w_2,w_3\}\\
&=&\{(v_1,w_2),(v_2,w_2),(v_2,w_3)\}-\{(v_1,w_2), (v_1,w_3),(v_2,w_3)\}\\
&&+\{(v_2,w_2),(v_3,w_2),(v_3,w_3)\}-\{(v_2,w_2),(v_2,w_3),(v_3,w_3)\}\\
(\partial_2^{\times}\circ\mu_2)(g_1)&=&\{(v_1,w_2),(v_2,w_2)\}+\{(v_2,w_2),(v_2,w_3)\}-\\
&&\{(v_1,w_2),(v_1,w_3)\}-\{(v_1,w_3),(v_2,w_3)\}+\{(v_2,w_2),(v_3,w_2)\}\\
&&+\{(v_3,w_2),(v_3,w_3)\}-\{(v_2,w_2),(v_2,w_3)\}-\{(v_2,w_3),(v_3,w_3)\}\\
&=&\{(v_1,w_2),(v_2,w_2)\}-\{(v_1,w_2),(v_1,w_3)\}-\{(v_1,w_3),(v_2,w_3)\}\\
&&+\{(v_2,w_2),(v_3,w_2)\}+\{(v_3,w_2),(v_3,w_3)\}-\{(v_2,w_3),(v_3,w_3)\}\\
&=&(\mu_1\circ(\partial\otimes\partial')_2)(g_1).
\end{eqnarray*}
\end{example}

\begin{proposition}\label{prop2}
Let $\mathcal{H},\mathcal{H}'$ be hypergraphs. Then there is a quasi-isomorphism
\begin{equation*}
  \mathrm{Inf}_{n}(R(\mathcal{H})_{\ast}\otimes R(\mathcal{H}')_{\ast})\rightarrow
   \mathrm{Inf}_{n}(R(\mathcal{H}\boxtimes \mathcal{H}')_{\ast})
\end{equation*}
\end{proposition}
\begin{proof}
By Lemma~\ref{le-6}, we have that
\begin{eqnarray*}
(\nu|_{\mathrm{Inf}_{n}(R(\mathcal{H}\boxtimes \mathcal{H}')_{\ast})_{\ast}})\circ(\mu|_{\mathrm{Inf}_{n}(R(\mathcal{H})_{\ast}\otimes R(\mathcal{H}')_{\ast})})_{\ast}=\mathrm{id}
\end{eqnarray*}
and
\begin{eqnarray}\label{eq-66}
(\nu|_{\mathrm{Inf}_{n}(R(\mathcal{H}\boxtimes \mathcal{H}')_{\ast})})_{\ast}
\end{eqnarray}
is surjective. Therefore, it leaves us to show (\ref{eq-66}) is injective. That is, let
\begin{eqnarray*}
x\in \mathrm{Inf}_{n}(R(\mathcal{H}\boxtimes \mathcal{H}')_{\ast})
\end{eqnarray*}
be a cycle such that
\begin{eqnarray*}
(\nu|_{\mathrm{Inf}_{n}(R(\mathcal{H}\boxtimes \mathcal{H}')_{\ast})})(x)
\end{eqnarray*} is a boundary in $\mathrm{Inf}_{n+1}(R(\mathcal{H})_{\ast}\otimes R(\mathcal{H}')_{\ast})\otimes C_{\ast}(\mathcal{K}_{\mathcal{H}'};R))$.
Then $x$ is a boundary in
\begin{eqnarray*}
\mathrm{Inf}_{n+1}(R(\mathcal{H}\boxtimes \mathcal{H}')_{\ast}).
\end{eqnarray*}

Since $\nu(x)$ is a boundary in $C_{\ast}(\mathcal{K}_{\mathcal{H}};R)\otimes C_{\ast}(\mathcal{K}_{\mathcal{H}'};R)$ and  there is a quasi-isomorphism
\begin{eqnarray*}
\xymatrix{
C_{\ast}(\mathcal{K}_{\mathcal{H}}\times  \mathcal{K}_{\mathcal{H}'};R) \ar@<0.5ex>[r]^{\nu~~~~~}& \ar@<0.5ex>[l]^{\mu~~~~~} C_{\ast}(\mathcal{K}_{\mathcal{H}};R)\otimes C_{\ast}(\mathcal{K}_{\mathcal{H}'};R),
}
\end{eqnarray*}
it follows that $x$ is a boundary in $C_{\ast}(\mathcal{K}_{\mathcal{H}}\times \mathcal{K}_{\mathcal{H}'};R)$. Hence, to show
\begin{eqnarray*}
(\nu|_{\mathrm{Inf}_{n}(R(\mathcal{H}\boxtimes \mathcal{H}')_{\ast})})_{\ast}
\end{eqnarray*}
is injective, it is sufficient to show $x$ is also a boundary in $R(\mathcal{H}\boxtimes \mathcal{H}')_{\ast}$.

Let $x=\sum\limits_i r_i\omega_i$ be an arbitrary element  in
\begin{eqnarray*}
\mathrm{Inf}_{n}(R(\mathcal{H}\boxtimes \mathcal{H}')_{\ast})
\end{eqnarray*}
where
\begin{equation*}
  \omega_i=\{(v_{\alpha(0)}^{(i)},w_{\beta(0)}^{(i)}),(v_{\alpha(1)}^{(i)},w_{\beta(1)}^{(i)}),\cdots,(v_{\alpha(p_{i}+q_i)}^{(i)},w_{\beta(p_{i}+q_i)}^{(i)})\}\in (\mathcal{H}\boxtimes \mathcal{H}')_n
\end{equation*}
for some $\sigma_i\in \mathcal{H}_{p_i}$ and $\tau_i\in\mathcal{H}'_{q_i}$, $p_i+q_i=n$, $r_i\in R$. There are two cases.


{\sc Case~1}. $p_i=0$ or $q_i=0$. Without loss of generality,  we assume that $p_i=n$ and  $q_i=0$. By Case 2 of Step 2 in Lemma~\ref{le-6},
\begin{eqnarray*}
x=\sum\limits_i x_i\boxtimes \{w_{\beta(0)}^{(i)}\}
\end{eqnarray*}
where $x_i$  is a linear combination of hyperedges of $\mathcal{H}_n$ and $\{w_{\beta(0)}^{(i)}\}$ is a $0$-hyperedge in $\mathcal{H}'$.

Since
\begin{eqnarray*}
\nu(x)&=&\sum\limits_i \nu(x_i\boxtimes \{w_{\beta(0)}^{(i)}\})\\
&=&\sum\limits_i x_i\otimes \{w_{\beta(0)}^{(i)}\}
 \end{eqnarray*}
is a boundary in $R(\mathcal{H})_{\ast}\otimes R(\mathcal{H}')_{\ast}$, it follows that $\{w_{\beta(0)}^{(i)}\}$ are the same for all $i$ and $\sum\limits_i x_i$ is a boundary in $R(\mathcal{H})_{n+1}$. Denote $\{w_{\beta(0)}^{(i)}\}$ as $\{w_{\beta(0)}\}$.  Then
\begin{eqnarray*}
\nu(x)=(\sum\limits_i x_i)\boxtimes \{w_{\beta(0)}\}
 \end{eqnarray*}
where $\sum\limits_i x_i$ is a boundary in $R(\mathcal{H})_{n+1}$. Hence $x$ is a boundary in $R(\mathcal{H}\boxtimes \mathcal{H}')_{n+1}$.

{\sc Case~2}. $p_i\geq 1$ and $q_i\geq 1$.  Since $x$ is  a boundary in $C_{\ast}(\mathcal{K}_{\mathcal{H}}\times \mathcal{K}_{\mathcal{H}'};R)$, it follows that there exists an element $\tilde x\in C_{\ast}(\mathcal{K}_{\mathcal{H}}\times \mathcal{K}_{\mathcal{H}'};R)$ such that $\partial \tilde x=x$. Let $\tilde x=\sum\limits_i b_i\tilde{\omega}_i$ where $\tilde{\omega}_i\in \mathcal{K}_{\mathcal{H}}\times \mathcal{K}_{\mathcal{H}'}$ and $b_i\in R$. Consider the following two subcases.

{\sc Subcase~2.1}. $\tilde{\omega}_i\in \mathcal{K}_{\mathcal{H}}\boxtimes \mathcal{K}_{\mathcal{H}'}$. Then
\begin{equation*}
  \tilde{\omega}_i=\{(v_{\alpha(0)}^{(i)},w_{\beta(0)}^{(i)}),(v_{\alpha(1)}^{(i)},w_{\beta(1)}^{(i)}),\cdots,(v_{\alpha(l+m)}^{(i)},w_{\beta(l+m)}^{(i)})\}
\end{equation*}
for some $\tilde{\sigma}_i=\{(v_{0}^{(i)},\cdots v_{l}^{(i)})\}\in \mathcal{K}_{\mathcal{H}},\tilde{\tau}_i=\{w_{0}^{(i)},\cdots,w_{m}^{(i)}\}\in \mathcal{K}_{\mathcal{H}'}$ such that either
\begin{equation*}
  \left\{
    \begin{array}{ll}
       \alpha(k+1)=\alpha(k)+1 \\
       \beta(k+1)=\beta(k)
    \end{array}
  \right.
  \end{equation*}
or
\begin{equation*}
  \left\{
    \begin{array}{ll}
       \alpha(k+1)=\alpha(k)\\
       \beta(k+1)=\beta(k)+1.
    \end{array}
  \right.
\end{equation*}

Since $\partial \tilde x=x \in R(\mathcal{H}\boxtimes \mathcal{H}')_{n}$, by similar discussion in Case 1 of Step 2 in Lemma~\ref{le-6},  all elements in $\tilde{\sigma}_i\boxtimes \tilde{\tau_i}$ are in $\tilde x$. Next, we will show that if there exists some $\tilde{\sigma}_i\notin \mathcal{H}$ or $\tilde{\tau}_i\notin\mathcal{H}'$ in $\tilde x$, $\tilde x$ can be substituted with an element in $R(\mathcal{H}\boxtimes \mathcal{H}')_{n+1}$, whose boundary is the same as $\partial^{\times}(\tilde x)$.

Suppose  only  $\tilde{\sigma}_i\in \mathcal{K}_{\mathcal{H}}\setminus\mathcal{H}$.
\smallskip

Step 1. Notice that $\mathrm{dim}(\mathcal{K}_{\mathcal{H}})=\mathrm{dim}(\mathcal{H})$. Since $\tilde{\sigma}_i\in \mathcal{K}_{\mathcal{H}}\setminus\mathcal{H}$, it follows that $\mathrm{dim}\tilde{\sigma}_i<\mathrm{dim}\mathcal{H}$ and there exists a hyperedge $\bar{\sigma}\in \mathcal{H}$ such that $\tilde{\sigma}_i$ is a face of  $\bar{\sigma}$, denoting as $\tilde{\sigma}_i=d_k(\bar\sigma)$ ($0\leq k\leq {l+1}$). Since $\tilde x\in R{\mathcal{H}\boxtimes \mathcal{H}'}_n$, it follows that $(-1)^{l-k+1}\bar{\sigma}\boxtimes \partial'(\tilde{\tau}_i)$ is a factor in $\tilde x$ such that $\tilde {\sigma}_i\boxtimes \partial'(\tilde{\tau}_i)$ is cancelled out within $\partial\tilde{x}$.

\smallskip
Step 2. We divided this step into Subcase~2.1.1 and Subcase~2.1.2.

{\sc Subcase~2.1.1}.  $\partial^{\times}\Big( \tilde{\sigma}_i\boxtimes \tilde{\tau}_i+(-1)^{l-k+1}\bar{\sigma}\boxtimes \partial'(\tilde{\tau}_i)\Big)\in R(\mathcal{H}\boxtimes\mathcal{H}')_n$.

Then $(\partial\tilde{\sigma}_i)\in \mathcal{H}_{l-1}$ and $d_j(\bar{\sigma})\in \mathcal{H}_l$ for all $j\not=k$. Let
\begin{eqnarray*}
\tilde{x}_1=-\sum\limits_{j\not=k}(-1)^{j-k}d_j(\bar{\sigma})\boxtimes \tilde{\tau}_i.
\end{eqnarray*}
Then $\tilde{x}_1\in R(\mathcal{H}\boxtimes\mathcal{H}')_{n+1}$.

Since
\begin{eqnarray*}
\partial^{\times}\Big(\partial\bar{\sigma}\boxtimes \tilde{\tau}_i+(-1)^{l+1}\bar{\sigma}\boxtimes \partial'(\tilde{\tau}_i)\Big)=0,
\end{eqnarray*}
it follows that
\begin{eqnarray*}
\partial^{\times}\tilde{x}_1=\partial^{\times}\Big( \tilde{\sigma}_i\boxtimes \tilde{\tau}_i+(-1)^{l-k+1}\bar{\sigma}\boxtimes \partial'(\tilde{\tau}_i)\Big).
\end{eqnarray*}

Let
\begin{eqnarray*}
\tilde X=\tilde x-\Big( \tilde{\sigma}_i\boxtimes \tilde{\tau}_i+(-1)^{l-k+1}\bar{\sigma}\boxtimes \partial'(\tilde{\tau}_i)\Big)+\tilde{x}_1.
\end{eqnarray*}
Then $\tilde X$ is an element in $R(\mathcal{H}\boxtimes\mathcal{H}')_{n+1}$ such that $\partial^{\times}(\tilde X)=\partial^{\times}( \tilde x)$. 

{\sc Subcase~2.1.2}. $\partial^{\times}\Big( \tilde{\sigma}_i\boxtimes \tilde{\tau}_i+(-1)^{l-k+1}\bar{\sigma}\boxtimes \partial'(\tilde{\tau}_i)\Big)\notin R(\mathcal{H}\boxtimes\mathcal{H}')_n$.

Since $\partial \tilde x\in R(\mathcal{H}\boxtimes\mathcal{H'})_n$, we assert that there exist some factors of
\begin{eqnarray*}
\tilde x-\Big( \tilde{\sigma}_i\boxtimes \tilde{\tau}_i+(-1)^{l-k+1}\bar{\sigma}\boxtimes \partial'(\tilde{\tau}_i)\Big)
\end{eqnarray*}
which can be added to $\Big( \tilde{\sigma}_i\boxtimes \tilde{\tau}_i+(-1)^{l-k+1}\bar{\sigma}\boxtimes \partial'(\tilde{\tau}_i)\Big)$ such that $d_j(\bar{\sigma})\notin \mathcal{H}_l$ ($j\not=k$) and $d_{k_i}(\tilde{\sigma}_i)\notin \mathcal{H}_{l-1}$ are cancelled out within $\partial \tilde x$.

The reasons are as follows.

For any  $d_{k_i}(\tilde{\sigma}_i)\notin \mathcal{H}_{l-1}$, there exists an unique  $j\not=k$ such that $d_{k_i}(\tilde{\sigma}_i)$ is a face of  $d_j(\bar{\sigma})$ and  $d_{k_i}(\tilde{\sigma}_i)\boxtimes \tilde{\tau}_i$   can be cancelled out in
\begin{eqnarray}\label{eq-ss}
\partial^{\times}\Big(\tilde{\sigma}_i\boxtimes \tilde{\tau}_i+(-1)^{l-k+1}\bar{\sigma}\boxtimes \partial'(\tilde{\tau}_i)+(-1)^{j-k}d_j(\bar{\sigma})\boxtimes \tilde{\tau}_i\Big).
\end{eqnarray}
Meanwhile,  $d_j(\bar{\sigma})\boxtimes \partial'(\tilde{\tau}_i)$ can be also cancelled out in (\ref{eq-ss}).

Hence, if  such $d_j(\bar{\sigma})\boxtimes \tilde{\tau}_i$ are all factors of $\tilde x$, then follows Step 2. Otherwise, follows Step 1. Thus, by repeating the process  with a finite number of times at most, we can have the assertion in Subcase 2.1.2 of  Step 2.

Therefore, in Subcase 2.1, we obtain that there exists an element $\tilde X$ such that $\tilde X\in R(\mathcal{H}\boxtimes\mathcal{H}')_{n+1}$ and $\partial^{\times}(\tilde X)=\partial^{\times}(\tilde x)=x$.

{\sc Subcase~2.2}.  $\tilde{\omega}_i\in \mathcal{K}_{\mathcal{H}}\times \mathcal{K}_{\mathcal{H}'}\setminus{\mathcal{K}_{\mathcal{H}}\boxtimes \mathcal{K}_{\mathcal{H}'}}$. Then for each
\begin{equation*}
  \tilde{\omega}_i=\{(v_{\alpha(0)}^{(i)},w_{\beta(0)}^{(i)}),\cdots,(v_{\alpha(k)}^{(i)},w_{\beta(k)}^{(i)}),(v_{\alpha(k+1)}^{(i)},w_{\beta(k+1)}^{(i)}),
  \cdots,(v_{\alpha(n+1)}^{(i)},w_{\beta(n+1)}^{(i)})\},
\end{equation*}
there exists some  $0\leq k<n$ such that $\alpha(k+1)=\alpha(k)+1$ and $\beta(k+1)=\beta(k)+1$. Since at least one of the terms
\begin{eqnarray*}
&& \{(v_{\alpha(0)}^{(i)},w_{\beta(0)}^{(i)}),\cdots,\widehat{(v_{\alpha(k)}^{(i)},w_{\beta(k)}^{(i)})},(v_{\alpha(k+1)}^{(i)},w_{\beta(k+1)}^{(i)}),
\cdots,(v_{\alpha(n+1)}^{(i)},w_{\beta(n+1)}^{(i)})\},\\
&& \{(v_{\alpha(0)}^{(i)},w_{\beta(0)}^{(i)}),\cdots,(v_{\alpha(k)}^{(i)},w_{\beta(k)}^{(i)}),\widehat{(v_{\alpha(k+1)}^{(i)},w_{\beta(k+1)}^{(i)})},
\cdots,(v_{\alpha(n+1)}^{(i)},w_{\beta(n+1)}^{(i)})\}
\end{eqnarray*}
is not in $\mathcal{K}_{\mathcal{H}}\boxtimes \mathcal{K}_{\mathcal{H}'}$, it follows that $\partial \tilde{\omega}_i\notin R(\mathcal{K}_{\mathcal{H}}\boxtimes \mathcal{K}_{\mathcal{H}'})_{\ast}$.

Notice that
\begin{eqnarray*}
R(\mathcal{H}\boxtimes \mathcal{H}')_{\ast}\subseteq R(\mathcal{K}_{\mathcal{H}}\boxtimes \mathcal{K}_{\mathcal{H}'})_{\ast}.
\end{eqnarray*}
Hence $\partial \tilde{\omega}_i\notin R(\mathcal{H}\boxtimes \mathcal{H}')_{\ast}$. Moreover, $\partial \tilde{\omega}_i$ can not be cancelled out by $\partial \tilde{\omega}_j$ ($j\not=i$) in $\partial \tilde x$. Then there is a factor in $\partial \tilde x$ which is in $R(\mathcal{H}\boxtimes \mathcal{H}')_{\ast}$, which contradicts $\partial \tilde x=x\in R(\mathcal{H}\boxtimes \mathcal{H}')_{\ast}$.

Summarising Subcase~2.1 and Subcase~2.2, we have that if $x=\partial\tilde x$ for some $\tilde x\in C_{\ast}(\mathcal{K}_{\mathcal{H}}\times \mathcal{K}_{\mathcal{H}'};R))$, then $\tilde x$ can be substituted with an element in $R(\mathcal{H}\boxtimes \mathcal{H}')_{n+1}$, whose boundary is the same as $\partial^{\times}(\tilde x)$, which implies $x$ is also a boundary in $R(\mathcal{H}\boxtimes\mathcal{H}'_{n+1})$.


Combining Case~1 and Case~2,  we have  that
\begin{eqnarray*}
(\nu|_{\mathrm{Inf}_{n}(R(\mathcal{H}\boxtimes \mathcal{H}')_{\ast},C_{\ast}(\mathcal{K}_{\mathcal{H}}\times \mathcal{K}_{\mathcal{H}'};R))})_{\ast}
\end{eqnarray*}
is injective. The proposition is proved.

\end{proof}

\begin{remark}
Quasi-isomorphism  is a weaker condition than  chain homotopy and the special chain homotopy (the contraction of chain complex) in
the Eilenberg-Zilber Theorem. 
For details of the quasi-isomorphism theorem, please refer
to \cite[Theorem~8.5]{rational}.
\end{remark}
\smallskip
\begin{theorem}[Algebraic K\"{u}nneth Formula](cf.\cite[Theorem~3B.5]{hatcher})\label{th-5.2}
Let $R$ be a principal ideal domain, and let $C_{\ast},C'_{\ast}$ be chain complexes of free $R$-module. Then there is a natural exact sequence
\begin{equation*}
 0\rightarrow \bigoplus\limits_{p+q=n}H_{p}(C_{\ast})\otimes_R H_{q}(C'_{\ast})\rightarrow H_{n}(C_{\ast}\otimes C'_{\ast})\rightarrow \bigoplus\limits_{p+q=n}\mathrm{Tor}_{R}(H_{p}(C_{\ast}),H_{q-1}(C'_{\ast}))\rightarrow 0.
\end{equation*}
\end{theorem}

\begin{theorem}[K\"{u}nneth Formula for Hypergraphs] 
\label{th-1.0}
Let $R$ be a principle ideal domain. Let $\mathcal{H},\mathcal{H}'$ be hypergraphs. Then there is a natural exact sequence
\begin{equation*}
  0\rightarrow \bigoplus\limits_{p+q=n}H_{p}(\mathcal{H})\otimes H_{q}(\mathcal{H}')\rightarrow H_{n}(\mathcal{H}\boxtimes\mathcal{H}')\rightarrow
  \bigoplus\limits_{p+q=n}\mathrm{Tor}_{R}(H_{p}(\mathcal{H}),H_{q-1}(\mathcal{H}'))\rightarrow 0.
\end{equation*}
\end{theorem}
\begin{proof}
By Theorem~\ref{th-5.2}, there is an exact sequence
\begin{equation*}
\begin{split}
    0\rightarrow \bigoplus\limits_{p+q=n}H_{p}(\mathcal{H}) \otimes H_{q}(\mathcal{H}') &\rightarrow H_{n}(\mathrm{Inf}_{\ast}(\mathcal{H})\otimes \mathrm{Inf}_{\ast}(\mathcal{H}')) \\
    &  \rightarrow \bigoplus\limits_{p+q=n}\mathrm{Tor}_{R}(H_{p}(\mathcal{H}),H_{q-1}(\mathcal{H}'))\rightarrow 0.
\end{split}
\end{equation*}
By Proposition~\ref{pr-2.1}, we have that
\begin{equation*}
  \mathrm{Inf}_{\ast}(\mathcal{H})\otimes \mathrm{Inf}_{\ast}(\mathcal{H}')= \mathrm{Inf}_{n}(R(\mathcal{H})_{\ast}\otimes R(\mathcal{H}')_{\ast},C_{\ast}(\mathcal{K}_{\mathcal{H}};R)\otimes C_{\ast}(\mathcal{K}_{\mathcal{H}'};R)).
\end{equation*}
By Proposition \ref{prop2}, we obtain that
\begin{equation*}
  H( \mathrm{Inf}_{n}(R(\mathcal{H})_{\ast}\otimes R(\mathcal{H}')_{\ast},C_{\ast}(\mathcal{K}_{\mathcal{H}};R)\otimes C_{\ast}(\mathcal{K}_{\mathcal{H}'};R)))\cong H(\mathrm{Inf}_{n}(R(\mathcal{H}\boxtimes \mathcal{H}')_{\ast},C_{\ast}(\mathcal{K}_{\mathcal{H}}\times \mathcal{K}_{\mathcal{H}'};R)).
\end{equation*}
Note that by Proposition~\ref{prop1}, $\mathcal{K}_{\mathcal{H}}\times \mathcal{K}_{\mathcal{H}'}=\mathcal{K}_{\mathcal{H}\boxtimes\mathcal{H}'}$, we have
\begin{equation*}
  H(\mathrm{Inf}_{\ast}(\mathcal{H})\otimes \mathrm{Inf}_{\ast}(\mathcal{H}'))\cong H(\mathcal{H}\boxtimes \mathcal{H}').
\end{equation*}
This implies our result.
\end{proof}

\begin{corollary}\label{co-4.2a}
Let $\mathcal{H}$ and $\mathcal{H}'$ be hypergraphs.  Let $\mathbb{F}$ be a field.   Then for each $n\geq 0$,  we have the isomorphism of $R$-modules
 \begin{eqnarray*}
 H_{n}(\mathcal{H}\boxtimes \mathcal{H}';\mathbb{F})\cong  (H_*(\mathcal{H};\mathbb{F})\otimes_\mathbb{F} H_*(\mathcal{H}';\mathbb{F}))_n.
\end{eqnarray*}
 \qed
\end{corollary}
The next example illustrates the above result.

\begin{example}\label{ex-6.1}
Let $\mathcal{H}=\{\{v_0\}, \{v_0,v_1\}, \{v_1,v_2\}, \{v_0,v_1,v_2\}\}$, $\mathcal{H'}=\{\{w_0\}, \{w_1\}, \{w_0,w_1\}\}$, and $\mathbb{F}$ is  a field. Then we  have that
\begin{equation*}
{\mathcal{H}\boxtimes \mathcal{H}'}= \left\{\begin{array}{cc}
\{(v_0,w_0)\}, & \{(v_0,w_1)\},\\
\{(v_0,w_0),(v_0,w_1)\}, & \{(v_0,w_0), (v_1, w_0)\},\\
\{(v_1,w_0),(v_2,w_0)\},& \{(v_0,w_1),(v_1,w_1)\},\\
\{(v_1,w_1),(v_2,w_2)\},& \{(v_0,w_0),(v_0,w_1),(v_1,w_1)\},\\
\{(v_0,w_0),(v_1,w_0),(v_1,w_1)\},&\{(v_1,w_0),(v_1,w_1),(v_2,w_1)\},\\
\{(v_1,w_0),(v_2,w_0),(v_2,w_1)\},&\{(v_0,w_0),(v_1,w_0),(v_2,w_0)\},\\
\{(v_0,w_1),(v_1,w_1),(v_2,w_1)\},&\{(v_0,w_0),(v_1,w_0),(v_2,w_0),(v_2,w_1)\},\\
\{(v_0,w_0),(v_1,w_0),(v_1,w_1),(v_2,w_1)\},&\{(v_0,w_0),(v_0,w_1),(v_1,w_1),(v_2,w_1)\}\\
\end{array}\right\}
\end{equation*}
According to definition of the embedded homology and Proposition~2.3 in \cite{h1}, we know that
\begin{eqnarray*}
 H_n(\mathcal{H};\mathbb{F})&=&H_n(\text{Inf}_{*}(\mathcal{H}))\\
 &=&\text{Ker}(\partial_n\mid_{C_n(\mathcal{H};\mathbb{F})})/(C_n(\mathcal{H};\mathbb{F}) \cap {\partial_{n+1}(C_{n+1}(\mathcal{H};\mathbb{F}))})\\
 H_0(\mathcal{H}\boxtimes\mathcal{H'};\mathbb{F})&=&\text{Ker}(\partial_0\mid_{\langle \{(v_0,w_0)\}, \{(v_0,w_1)\}\rangle})/\langle\{(v_0,w_1)\}-\{(v_0,w_0)\}\rangle\\
&=&\mathbb{F}\\
H_0(\mathcal{H};\mathbb{F})&=&\text{Ker}(\partial_0\mid_{\langle \{v_0\}\rangle})/\langle \{v_1\}-\{v_0\}, \{v_2\}-\{v_1\} \rangle\cap{\langle \{v_0\} \rangle}\\
&=&\mathbb{F}\\
H_0(\mathcal{H'};\mathbb{F})&=&\text{Ker}(\partial_0\mid_{\langle \{w_0\}, \{w_1\}\rangle})/\langle \{w_1\}-\{v_0\} \rangle\\
&=&\mathbb{F}\\
H_p(\mathcal{H}\boxtimes\mathcal{H'};\mathbb{F})&=&0, H_p(\mathcal{H};\mathbb{F})=H_p(\mathcal{H'};\mathbb{F})=0 ~\text{for~any}~ p \geq 1.
\end{eqnarray*}
Thus
\begin{eqnarray*}
 H_{n}(\mathcal{H}\boxtimes \mathcal{H}';\mathbb{F})\cong  (H_*(\mathcal{H};\mathbb{F})\otimes_\mathbb{F} H_*(\mathcal{H}';\mathbb{F}))_n.
\end{eqnarray*}
\end{example}

\subsection{Further Discussions}\label{sec-6}

In this section,  we  discuss briefly that our proof for the K\"{u}nneth formula for hypergraphs is applicable to give an alternative proof for the K\"{u}nneth formula for the path homology of digraphs.

\smallskip

A digraph is a pair $(V,E)$ where $V$ is the vertex set and $E$ is a subset of $V\times V$.  For any $(a,b)\in E$, we write $a\to b$ and call it a directed edge.
In \cite{yau4,yau1,yau2,3},   the professors  A. Grigor'yan,  Y. Lin,  Y. Muranov, V. Vershinin and S.T. Yau  defined the path complex  for a digraph where the
allowed paths go along the arrows of the directed edges; and  with the help of path complex, the path homology for a digraph is defined and studied.
The K\"{u}nneth formula for the path homology (with field coefficients) of digraphs is proved in \cite[Section~7]{yau1} and \cite{3}.
We can generalize it and obtain a K\"{u}nneth formula for the path homology with coefficients in a general principal ideal domain $R$.

We regard the set of allowed paths in a digraph as a graded subset of certain simplicial set,
and regard the path complex as a graded abelian subgroup of certain chain complex.  By Theorem~\ref{th-2.1}, the Eilenberg-Zilber Theorem
of simplicial sets, and a similar argument of the K\"{u}nneth formula for hypergraphs (Proof of Theorem~\ref{th-1.0}),
we would obtain an alternative proof for the K\"{u}nneth formula for the path homology,
and generalize the coefficients  from a field to general principal ideal domains.




\medskip
\noindent {\bf Acknowledgement}. The authors would like to thank Prof. Yong Lin and Prof. Jie Wu  for their supports,  discussions and encouragements.
The authors also would like to express their deep gratitude to the reviewer(s) for their careful reading, valuable comments, and helpful suggestions.

The first author is supported by the Youth Fund of Hebei Provincial Department of Education (QN2019333),
the Natural Fund of Cangzhou Science and Technology Bureau (No.197000002) and a Project of Cangzhou Normal University (No.xnjjl1902).
The second author is supported by the Postdoctoral International Exchange Program of
China 2019 project from The Office of China Postdoctoral Council, China Postdoctoral Science Foundation.

\bigskip

Chong Wang  (for correspondence)

 Address: $^1$School of Mathematics, Renmin University of China, 100872, China.

 $^2$School of Mathematics and Statistics, Cangzhou Normal University, 061000 China .

 e-mail:  wangchong\_618@163.com

  \medskip

 Shiquan Ren  %

 Address: Yau  Mathematical Sciences Center, Tsinghua University,  100084 China.

 e-mail: srenmath@126.com

\medskip

Jian Liu

Address: School of Mathematical Sciences and LPMC, Nankai University, 300071 China.

e-mail: liujian2@mail.nankai.edu.cn

 \end{document}